\documentclass[12pt]{amsart}
\usepackage{graphicx}
\usepackage{amsmath,amssymb}
\usepackage{array}
\usepackage{float}
\usepackage{wrapfig}
\usepackage{booktabs}
\usepackage{ascmac}
\usepackage{fancybox}
\usepackage{tikz}
\usepackage{bm}
\usepackage{pxpgfmark}
\usepackage{ascmac}
\usepackage[all]{xy}
\usetikzlibrary{patterns}
\usetikzlibrary{intersections, calc}
\usetikzlibrary{quotes,angles}
\usetikzlibrary {arrows.meta}
\usetikzlibrary {bending}
\makeatletter
\newcommand{\xequal}[2][]{\ext@arrow 0055{\equalfill@}{#1}{#2}}
\def\equalfill@{\arrowfill@\Relbar\Relbar\Relbar}
\makeatother
\newcommand{\maru}[1]{\raise0.2ex\hbox{\textcircled{{\rm Stab}(G)riptsize{#1}}}}

 \def\NZQ{\mathbb}               % the font for N,Z,Q,R,C

 \def\RR{{\NZQ R}}

 \def\Jc{{\mathcal J}}

 \def\Gc{{\mathcal G}}
 \def\Fc{{\mathcal F}}
 \def\Cc{{\mathcal C}}

 \def\Pc{{\mathcal P}}
 \def\Mc{{\mathcal M}}
 \def\Hc{{\mathcal H}}
 \def\Gc{{\mathcal G}}

 \def\Oc{{\mathcal O}}

 \def\eb{{\mathbf e}}

 \def\opn#1#2{\def#1{\operatorname{#2}}} % to make operators
 \opn\chara{char} \opn\length{\ell} \opn\pd{pd} \opn\rk{rk}
 \opn\projdim{proj\,dim} \opn\injdim{inj\,dim} \opn\rank{rank}
 \opn\depth{depth} \opn\grade{grade} \opn\height{height}
 \opn\embdim{emb\,dim} \opn\codim{codim}
 
 \opn\Tr{Tr} \opn\bigrank{big\,rank}
 \opn\superheight{superheight}\opn\lcm{lcm}
 \opn\trdeg{tr\,deg}%\emph{
 \opn\reg{reg} \opn\lreg{lreg} \opn\ini{in} \opn\lpd{lpd}
 \opn\size{size} \opn\sdepth{sdepth}
 \opn\link{link}\opn\fdepth{fdepth}\opn\lex{lex}
 \opn\tr{tr}
 \opn\type{type}
 \opn\gap{gap}
 \opn\arithdeg{arith-deg}
 \opn\revlex{revlex}
 \opn\cut{cut}
 \opn\div{div} \opn\Div{Div} \opn\cl{cl} \opn\Cl{Cl}
 \opn\Spec{Spec} \opn\Supp{Supp} \opn\supp{supp} \opn\Sing{Sing}
 \opn\Ass{Ass} \opn\Min{Min}\opn\Mon{Mon}
 \opn\Ann{Ann} \opn\Rad{Rad} \opn\Soc{Soc}
 \opn\Im{Im} \opn\Ker{Ker} \opn\Coker{Coker} \opn\Am{Am}
 \opn\Hom{Hom} \opn\Tor{Tor} \opn\Ext{Ext} \opn\End{End}
 \opn\Aut{Aut} \opn\id{id}
 
 \opn\nat{nat}
 \opn\pff{pf}%   \pf exists already
 \opn\Pf{Pf} \opn\GL{GL} \opn\SL{SL} \opn\mod{mod} \opn\ord{ord}
 \opn\Gin{Gin} \opn\Hilb{Hilb}\opn\sort{sort}
 \opn\PF{PF}\opn\Ap{Ap}
 \opn\mult{mult}
 \opn\bight{bight}
 \opn\aff{aff}
 \opn\relint{relint} \opn\st{st}
 \opn\lk{lk} \opn\cn{cn} \opn\core{core} \opn\vol{vol}  \opn\inp{inp} \opn\nilpot{nilpot}
 \opn\link{link} \opn\star{star}\opn\lex{lex}\opn\set{set}
 \opn\width{wd}
 \opn\Fr{F}
 \opn\QF{QF}
 \opn\G{G}
 \opn\type{type}\opn\res{res}
 \opn\conv{conv}
 \opn\sk{sk}
 \opn\conv{conv}
 \opn\Deg{Deg}
 \opn\Sym{Sym}
 \opn\gr{gr}
 
 \def\pot#1#2{#1[\kern-0.28ex[#2]\kern-0.28ex]}

 \opn\dirlim{\underrightarrow{\lim}}
 \opn\inivlim{\underleftarrow{\lim}}
 \def\Implies{\ifmmode\Longrightarrow \else
         \unskip${}\Longrightarrow{}$\ignorespaces\fi}
 \def\implies{\ifmmode\Rightarrow \else
         \unskip${}\Rightarrow{}$\ignorespaces\fi}
 \def\iff{\ifmmode\Longleftrightarrow \else
         \unskip${}\Longleftrightarrow{}$\ignorespaces\fi}

 \let\:=\colon
 \newtheorem{Theorem}{Theorem}[section]
 \newtheorem{Lemma}[Theorem]{Lemma}
 
 \newtheorem{Proposition}[Theorem]{Proposition}

 \let\epsilon\varepsilon
 \let\kappa=\varkappa
 \opn\dis{dis}
 \def\pnt{{\raise0.5mm\hbox{\large\bf.}}}
 
 \opn\Lex{Lex}

\begin{document}

\title{Simplex faces of order and chain polytopes}
\author {Aki Mori}

\address{Aki Mori, 
Learning center, Institute for general education, 
Setsunan University, Neyagawa, Osaka, 572-8508, Japan}
\email{aki.mori@setsunan.ac.jp}
\dedicatory{ }

\begin{abstract}
It will be proved that a $k$-clique in the $1$-skeleton of either the order polytope or the chain polytope corresponds to the $(k-1)$-face, which is a simplex, in each polytope. These results generalize the known explicit descriptions of edges and triangular $2$-faces of each polytope. 
\end{abstract}
\maketitle

\section{Introduction}
The order polytope $\Oc(P)$ and chain polytope $\Cc(P)$ of a finite poset (partially ordered set) $P$ were introduced by Stanley \cite{Sta}.
These polytopes represent significant classes of $n$-dimensional convex polytopes and have been extensively studied in combinatorics and commutative algebra.
In \cite{HL2}, a characterization of posets $P$ for which $\mathcal{O}(P)$ and $\mathcal{C}(P)$ are unimodularly equivalent was given.
Our focus is on the face structure of each polytope.
In each of $\Oc(P)$ and $\Cc(P)$, the explicit descriptions of vertices and facets were obtained by \cite{Sta}, edges by \cite{HL}, and triangular $2$-faces by \cite{M}, respectively.
In \cite[Corollary 1]{M}, it is proved that a triangle in the $1$-skeleton of either $\Oc(P)$ or $\Cc(P)$ corresponds to the triangular $2$-face of each polytope.
In the present paper, we aim to extend this statement to any dimensions.
Namely, we will prove that a $k$-clique in the 1-skeleton of $\Oc(P)$ or $\Cc(P)$ corresponds to the $(k-1)$-simplex face of each polytope.
Note that if a face of a convex polytope is a simplex, then we call it a {\em simplex face} throughout this paper.

We introduce the definition and notation that will be used in the subsequent sections. 
Let $\Pc$ be a convex polytope and $V(\Pc)$ be the set of its vertices. 
The {\em $1$-skeleton} of $\Pc$ is the finite simple graph $\sk(\Pc)$ on $V(\Pc)$ whose edges are $\{v, w\}$, where $v, w \in V(\Pc)$ and $\conv(\{v, w\})$ is an edge of $\Pc$. 
A {\em clique} of $\sk(\Pc)$ is a complete subgraph of $\sk(\Pc)$. 
In particular, a clique consisting of $k$ vertices is called a {\em k-clique}. 
A polytope $\Pc$ is {\it k-neighborly} if every subset of at most $k$ of its vertices defines a face of $\Pc$. 
Thus, a subset $C \subset V(\Pc)$ is a clique of $\sk(\Pc)$ if and only if the convex hull $\conv(C) \subset \Pc$ is a $2$-neighborly subpolytope of $\Pc$. 
In general, $\conv(C)$ may not be a face of $\Pc$. 
In the subsequent sections, we will prove that when $\Pc$ is either $\Oc(P)$ or $\Cc(P)$, $\conv(C)$ is a face of $\Pc$ and, furthermore, that $\conv(C)$ is a simplex.

\section{Order polytopes and their simplex faces}
Let $P = \{p_1, \ldots, p_d\}$ be a finite partially ordered set.  
A {\em poset ideal} of $P$ is a subset $I \subset P$ such that if $p_i \in I$ and $p_j \leq p_i$, then $p_j \in I$.  
In particular, the empty set $\emptyset$ as well as $P$ itself is a poset ideal of $P$.  
Let $\Jc(P)$ denote the set of poset ideals of $P$. 
Let $\eb_1, \ldots, \eb_d$ denote the canonical unit coordinate vectors of $\RR^d$.  
For each $I \in \Jc(P)$, one introduces $\rho(I) = \sum_{p_i \in I} \eb_i \in \RR^d$.  
Thus, $\rho(\emptyset)$ is the origin of $\RR^d$.  
The {\em order polytope} \cite{Sta} of $P$ is the convex hull $\Oc(P) = \conv(\{\rho(I) : I \in \Jc(P)\})$ in $\RR^d$.  
It is known that $\dim(\Oc(P)) = d$, and according to \cite[Corollary1.3]{Sta}, $V(\Oc(P)) = \{\rho(I) : I \in \Jc(P)\}$.  
Furthermore, \cite[Lemma 4]{HL} says that $\conv(\{\rho(I), \rho(I')\})$ with $I, I' \in \Jc(P)$ is an edge of $\Oc(P)$ if and only if $I \subset I'$ and $I' \setminus I$ is connected in $P$.  
In other words, the $1$-skeleton $\sk(\Oc(P))$ of $\Oc(P)$ is the finite simple graph on $V(\Oc(P))$ whose edges are those $\{\rho(I), \rho(I')\}$, where $I, I' \in \Jc(P)$, for which $I \subset I'$ and $I' \setminus I$ is connected in $P$.
This statement provides an explicit description of the edges of $\Oc(P)$.

\begin{Lemma}\label{clique}
A subset $C = \{\rho(I_0), \rho(I_1), \ldots, \rho(I_q)\} \subset V(\Oc(P))$ is a clique of $\sk(\Oc(P))$ if and only if $I_0 \subset I_1 \subset \cdots \subset I_q$ and each $I_j \setminus I_i$ with $i < j$ is connected in $P$. 
\end{Lemma}

\begin{proof}
A subset $C = \{\rho(I_0), \rho(I_1), \ldots, \rho(I_q)\} \subset V(\Oc(P))$ is a clique of $\sk(\Oc(P))$ if and only if each $\conv(\{\rho(I_i), \rho(I_j)\})$ with $i \neq j$ is an edge of $\Oc(P)$.  In other words, one has $I_i \subset I_j$ and $I_j \setminus I_i$ is connected in $P$.  Thus, rearranging the subscripts, if necessary, will produce the desired result.  
\end{proof}

The description of the facets of $\Oc(P)$ is obtained as follows.

\begin{Lemma}[\cite{Sta}]\label{order facet}
The facets of $\Oc(P)$ are given by the hyperplane whose defining equation is as follows:
\begin{itemize}
\item 
$x_{i}=0$, where $p_{i} \in P$ is maximal;
\item
$x_{j}=1$, where $p_{j} \in P$ is minimal;
\item
$x_{i}=x_{j}$, where $p_{j}$ covers $p_{i}$ in $P$.
\end{itemize}
\end{Lemma}

We use these lemmas to prove the following key proposition of this section.

\begin{Proposition}
\label{keyORDER}
Suppose that a subset $C \subset V(\Oc(P))$ is a clique of $\sk(\Oc(P))$.  Then $\conv(C)$ is a face of $\Oc(P)$. 
\end{Proposition}

\begin{proof}
Let $C = \{\rho(I_0), \rho(I_1), \ldots, \rho(I_q)\} \subset V(\Oc(P))$ be a clique of $\sk(\Oc(P))$ and suppose that $I_0 \subset I_1 \subset \cdots \subset I_q$ and each $I_j \setminus I_i$ with $i < j$ is connected in $P$, by Lemma \ref{clique}.  
Let $1 \leq i_0 \leq q$.  
Since $I_{i_0} \setminus I_{i_0-1}$ is connected, it follows from Lemma \ref{order facet} that 
\[
\Fc_{i_0} = \Oc(P) \cap \left(\,\bigcap_{p_i, \, p_j \in I_{i_0} \setminus I_{i_0-1}} \,  \Hc_{i,j}\right)
\]   
is a face of $\Oc(P)$, where $\Hc_{i,j}$ is the hyperplane 
\[
\Hc_{i,j} = \{(x_1, \ldots, x_d) \in \RR^d : x_i = x_j\}
\]
of $\RR^d$.  Furthermore,
\[
\Fc_{0} = \Oc(P) \cap \left(\,\bigcap_{p_i \in I_{0}} \Gc_{i}\right), \, \, \, \, \, 
\Fc_{q+1} = \Oc(P) \cap \left(\,\bigcap_{p_i \not\in I_{q}} \Gc'_{i}\right)
\]
are faces of $\Oc(P)$, where $\Gc_{i}$ and $\Gc'_{i}$ are the hyperplanes 
\[
\Gc_{i} = \{(x_1, \ldots, x_d) \in \RR^d : x_i = 1\}, \, \, \, \, \, 
\Gc'_{i} = \{(x_1, \ldots, x_d) \in \RR^d : x_i = 0\}
\]
of $\RR^d$.  It then follows that 
\[
\Fc = \bigcap_{0 \leq i \leq q+1} \Fc_{i}
\]
is a face of $\Oc(P)$.  
We claim $\Fc = \conv(C)$.  

Since $\Fc$ is a face of $\Oc(P)$, one has $\Fc = \conv(\Fc \cap V(\Oc(P))$.  
Clearly, each $\rho(I_i)$ belongs to $\Fc$.  
In order to show $\Fc = \conv(C)$, one must prove that $\rho(I) \not\in \Fc$ if $I \in \Jc(P)$ and if $I \neq I_i$ for each $i = 0,1,\ldots,q$.  
If $I \not\subset I_q$ and if $p_a \in I \setminus I_q$, then $\rho(I) \not\in \Gc'_a$ and $\rho(I) \not\in \Fc_{q+1}$.  
If $I_0 \not\subset I$ and if $p_a \in I_0 \setminus I$, then $\rho(I) \not\in \Gc_a$ and $\rho(I) \not\in \Fc_{0}$.  
Let $I_0 \subset I \subset I_q$ and suppose that $\rho(I) \in \Fc$.  
It follows that, for each $1 \leq i \leq q$, one has either $I_i \setminus I_{i-1} \subset I$ or $(I_i \setminus I_{i-1}) \cap I = \emptyset$.  
Let $1 \leq i_0 < q$ denote the biggest integer for which $I_{i_0} \subset I$.  
Then $(I_{i_0+1} \setminus I_{i_0}) \cap I = \emptyset$.
Since  
\[
I_q = I_0 \cup (I_1 \setminus I_{0}) \cup \cdots \cup (I_{q} \setminus I_{q-1}), 
\]  
one can choose the smallest integer $j_0 > i_0$ for which $(I_{j_0+1} \setminus I_{j_0}) \cap I \neq \emptyset$.  
Then $I_{j_0+1} \setminus I_{j_0} \subset I$ and $(I_{j_0} \setminus I_{j_0-1}) \cap I = \emptyset$. 
Since $I_{j_0+1} \setminus I_{j_0-1}$ is connected, there is $p_a \in I_{j_0+1} \setminus I_{j_0}$ and $p_b \in I_{j_0} \setminus I_{j_0-1}$ for which $\{p_a, p_b\}$ is an edge of the Hasse diagram of $P$.
Since $p_a \in I$ and $p_b \not\in I$, one has $p_a < p_b$.  
However, since $I_{j_0}$ is a poset ideal of $P$ and $p_a \not\in I_{j_0}$, one has $p_b < p_a$, a contradiction.  
Thus $\rho(I) \not\in \Fc$, as desired.  
\end{proof}

In general, $\conv(C)$ is a $2$-neighborly polytope and not necessarily a simplex. However, according to Proposition \ref{keyORDER} and the following lemma, it immediately turns out to be a simplex. We continue with the main theorem of this paper.

\begin{Lemma}\cite[p.123]{Gru}\label{neighborly}
If $\Pc$ is a $k$-neighborly $d$-polytope and $k > \lfloor \frac{1}{2}d \rfloor$ then $\Pc$ is a $d$-simplex.    
\end{Lemma}

\begin{Theorem}\label{main1}
Suppose that a subset $C \subset V(\Oc(P))$ is a $k$-clique of $\sk(\Oc(P))$.  
Then $\conv(C)$ is a $(k-1)$-simplex face of $\Oc(P)$. 
\end{Theorem}

\begin{proof}
From Proposition \ref{keyORDER}, $\conv(C)$ is a face of $\Oc(P)$.
If $C' \subset C$ is a $(k - \ell)$-clique of $\sk(\Oc(P))$, where $1 \leq \ell \leq k-1$, then $\conv(C')$ is similarly a face of $\Oc(P)$.
Since $\conv(C)$, $\conv(C')$ are faces of $\Oc(P)$ and $\conv(C') \subset \conv(C)$, then $\conv(C')$ is a face of $\conv(C)$.
Namely, every subset of at most $(k-1)$ vertices of $\conv(C)$ defines a face of $\conv(C)$.
Thus, $\conv(C)$ is a $(k-1)$-neighborly polytope.
Furthermore, since $\dim(\conv(C)) \leq k-1$, and if $\dim(\conv(C)) < k-1$, then $\conv(C)$ cannot be $(k-1)$-neighborly.
Consequently, $\dim(\conv(C)) = k-1$. 
It follows from Lemma \ref{neighborly} that $\conv(C)$ is a $(k-1)$-simplex.
\end{proof}

\section{Chain polytopes and their simplex faces}
Let $P = \{p_1,\ldots,p_d\}$ be a finite partially ordered set.
An {\it antichain} of $P$ is a subset $W \subset P$ such that $x_i$ and $x_j$ belonging to $W$ with $i \neq j$ are incomparable.
In particular, the empty set $\emptyset$ as well as each $\{p_i\}$ is an antichain of $P$.
The {\it chain polytope} \cite{Sta} of $P$ is the convex hull $\Cc(P) = \conv (\{\rho(W) : W \mbox{ is an antichain of }P\})$ in $\RR^d$.
Note that, the definition of $\rho(W)$ is the same as in the previous section.
It is known that $\dim(\Cc(P)) = d$, and according to \cite[Theorem 2,2]{Sta}, $V(\Cc(P)) = \{\rho(W) : W \mbox{ is an antichain of }P\}$. 
Furthermore, \cite[Lemma 5]{HL} says that $\conv(\{\rho(W), \rho(W')\})$, where $W$, $W'$ are antichains of $P$ is an edge of $\Cc(P)$ if and only if $W \triangle W'$ is connected in $P$. 
Note that, $W \triangle W' = (W \setminus W') \cup (W' \setminus W)$.  
In other words, the $1$-skeleton $\sk(\Cc(P))$ of $\Cc(P)$ is the finite simple graph on $V(\Cc(P))$ whose edges are those $\{\rho(W), \rho(W')\}$, where $W, W'$ are antichains of $P$, for which $W \triangle W'$ is connected in $P$.
From this explicit description of the edges of $\Cc(P)$, the following lemma is clear.

\begin{Lemma}\label{chainCLIQUE}
A subset $C = \{\rho(W_0), \rho(W_1), \ldots, \rho(W_q)\} \subset V(\mathcal{C}(P))$ is a clique of $\sk(\mathcal{C}(P))$ if and only if each $W_i\bigtriangleup W_j$ with $i \neq j$ is connected in $P$. 
\end{Lemma}

The description of the facets of $\Cc(P)$ is obtained as follows.

\begin{Lemma}[\cite{Sta}]\label{chain facet}
The facets of $\Cc(P)$ are given by the hyperplane whose defining equation is as follows:
\begin{itemize}
\item 
$x_{i}=0$, for all $p_{i}\in P$;
\item
$\sum_{\ell = 1}^k x_{i_{\ell}} = 1$, where $p_{i_{1}}<\cdots<p_{i_{k}}$
 is a maximal chain of $P$.
\end{itemize}
\end{Lemma}

In the proof of \cite[Theorem 2.1]{HLSS}, the following claim has been obtained.

\begin{Lemma}\label{antichain}
Let $A$ and $B$ be antichains of $P$ such that $A \neq B$, and $A \triangle B$ is connected in $P$.
Then, $A$ and $B$ satisfy either (i) $B \subset A$ or (ii) $b < a$ whenever $a \in A$ and $b \in B$ are comparable.
\end{Lemma}

We use these lemmas to prove the following key proposition of this section.

\begin{Proposition}
\label{keyCHAIN}
Suppose that a subset $C \subset V(\mathcal{C}(P))$ is a clique of $\sk(\mathcal{C}(P))$.
Then $\conv(C)$ is a face of $\mathcal{C}(P)$.
\end{Proposition}

\begin{proof}
Let $C = \{\rho(W_0), \rho(W_1), \ldots, \rho(W_q)\} \subset V(\mathcal{C}(P))$ be a clique of $\sk(\mathcal{C}(P))$, where each $W_i \subset P$ is an antichain of $P$.
Lemma \ref{chainCLIQUE} says that each $W_i\bigtriangleup W_j$ with $i \neq j$ is connected in $P$.  
Let $\Mc(C)$ denote the set of all maximal chains $\Cc$ of $P$ with $\Cc \cap W_i \neq \emptyset$ for each $i$.
Lemma \ref{chain facet}
%The description \cite{Sta} of the facets of %$\mathcal{C}(P)$ 
guarantees that
\[
\Fc = \mathcal{C}(P) \cap \left(\bigcap_{\Cc \in \Mc(C)} \Hc_\Cc \right)
\]
is a face of $\mathcal{C}(P)$, where $\Hc_\Cc$ is the hyperplane
\[
\Hc_\Cc = \{(x_1, \ldots, x_d) \in \RR^d: \sum_{p_i \in \Cc} x_i = 1\}
\]
of $\RR^d$.  Furthermore,
\[
\Fc' = \mathcal{C}(P) \cap \left(\,\bigcap_{p_i \, \not\in \,W_0 \cup W_1 \cup \cdots \cup W_q} \Hc'_i \right)
\]
is a face of $\mathcal{C}(P)$, where $\Hc'_i$ is the hyperplane 
\[
\Hc'_i = \{(x_1, \ldots, x_d) \in \RR^d : x_i = 0 \}
\]
of $\RR^d$. 
Let $\Gc = \Fc \cap \Fc'$, which is a face of $\mathcal{C}(P)$.
We claim $\Gc = \conv(C)$.

Since $\Gc$ is a face of $\Cc(P)$, one has $\Gc = \conv(\Gc \cap V(\Cc(P)))$.
Clearly, each $\rho(W_i)$ belongs to $\Gc$.
Let $A$ be an antichain of $P$, where $A \neq W_i$.
If $A \not\subset W_0 \cup \cdots \cup W_q$, then $\rho(A) \not\in \mathcal{F}'$. 
Hence, we assume $A \subset W_0 \cup \cdots \cup W_q$.
Since each $W_i\bigtriangleup W_j$ with $i \neq j$ is connected, $W_i \not\subset A$.
Furthermore, if there exists an $i$ such that $A \subset W_i$, then $\rho(A) \notin \Fc$, so we set $A \not\subset W_i$.
From these claim, for each $W_i$, it follows that either $A \cap W_i = \emptyset$ or both $A \setminus W_i \neq \emptyset$ and $W_i \setminus A \neq \emptyset$ hold.
For simplicity, among all $W_i$, we define $W_0, W_1, \ldots, W_s$ to be those that satisfy
\[W_i \cap A \neq \emptyset,\;\;\;\mbox{and}\;\;\; W_i \not\subset \bigcup_{\substack{ 0 \leq j \leq q\\i \neq j}} W_j,\]
where $s \leq q$.
Suppose that for $W_0, W_1, \ldots, W_s$, every pair of elements $\{p_i, p_j\}$ with $i \neq j$ from $p_0 \in W_0, p_1 \in W_1, \ldots, p_s \in W_s$ is comparable. Then, Lemma \ref{antichain} allows us to assume that $p_0 < p_1 < \cdots < p_s$ holds.
We define $A_i = (A \cap W_i) \setminus \bigcup_{j=0}^{i-1} (A \cap W_j)$ for $0 \leq i \leq s$. Then, we have $A = \bigcup_{i=0}^s A_i$ and $A_i \cap A_j = \emptyset$ for $i \neq j$.
Furthermore, when $A = \{p_{i_0},p_{i_1},\ldots,p_{i_m}\}$, we define $\gamma(A)$ as follows:
\[\gamma(A) = \sum_{k=0}^m \gamma(p_{i_k}),\quad \mbox{where}\quad\gamma(p_{i_k})=|\{\Cc \in \Mc(C) : p_{i_k} \in \Cc\}|.\]
In subsequent parts of this section, we will show that $\gamma(A) < |\Mc(C)|$.
For the sake of convenience, we set $A_0 = A''_0$.
We define $A''_{\ell + 1}$ in sequence as follows for $\ell = 0,1,\ldots, s-1$:
\[
A''_{\ell + 1} = \left\{
\begin{array}{ll}
\{p_{i_k} \in W_{\ell + 1} : p_{i_k} > p_{i_j},  p_{i_j} \in A''_{\ell} \setminus W_s \} \cup A_{\ell + 1} &  \mbox{if}\;\; A''_{\ell} \not\subset W_s,\\
A_{\ell + 1} & \mbox{if}\;\; A''_{\ell} \subset W_s.
\end{array}
\right.
\]
We note the following:
\begin{itemize}
\item
When $A''_{\ell} \not\subset W_s$, we set $A'_{\ell + 1} = \{p_{i_k} \in W_{\ell + 1} : p_{i_k} > p_{i_j}, p_{i_j} \in A''_{\ell} \setminus W_s \}$. Then, $A'_{\ell + 1} \cap A_{\ell + 1} = \emptyset$ and $A'_{\ell + 1} \cup A_{\ell + 1} \subset W_{\ell + 1}$ holds. Since $W_{\ell} \triangle W_{\ell + 1}$ is connected, it follows that $\gamma(A''_{\ell} \setminus W_s) < \gamma(A'_{\ell + 1})$. Therefore,
\[\gamma(A''_{\ell}) = \gamma(A''_{\ell} \cap W_s) + \gamma(A''_{\ell} \setminus W_s) < \gamma(A''_{\ell} \cap W_s) + \gamma(A'_{\ell + 1})\]
holds;
\item
When $A''_{\ell} \subset W_s$, we have  
\[\gamma(A''_{\ell}) = \gamma(A''_{\ell} \cap W_s).\]
In particular, if $\ell = s-1$, then $A''_s \subset W_s$.
Therefore, we can denote $\gamma(A''_s) = \gamma(A''_s \cap W_s)$;
\item
We have 
\[\bigcup_{\ell = 0}^s (A''_{\ell} \cap W_s) \subset W_s \quad\mbox{and}\quad (A''_i \cap W_s) \cap (A''_j \cap W_s) = \emptyset \quad\mbox{for}\quad i \neq j.\]
\end{itemize}
By using these sequentially, we establish the following: 
\[\gamma(A) = \sum_{\ell = 0}^{s} \gamma(A_{\ell}) < \sum_{\ell = 0}^{s} \gamma(A''_{\ell} \cap W_s) \leq  \gamma(W_s) = |\Mc(C)|.\]
Consequently, we have $\gamma(A) < |\Mc(C)|$.
This leads us to the conclusion that $\rho(A) \notin \Fc$, as desired.
\end{proof}

We can proceed with the same argument as in the previous section. 
Namely, it follows from Proposition \ref{keyCHAIN} and Lemma \ref{neighborly} that $\conv(C)$ is a simplex, thus establishing the other main theorem of this paper.
The proof is exactly the same as in Theorem \ref{main1}.

\begin{Theorem}
Suppose that a subset $C \subset V(\mathcal{C}(P))$ is a $k$-clique of $\sk(\mathcal{C}(P))$.
Then $\conv(C)$ is a $(k-1)$-simplex face of $\mathcal{C}(P)$. 
\end{Theorem}

\section*{Acknowledgement}
The author extends sincere gratitude to Takayuki Hibi for providing many useful discussions and ideas, and for the opportunity to write this paper.

\end{document}